\documentclass[a4paper, 12pt]{amsart}
\usepackage{amsmath}
\usepackage{amsfonts}
\usepackage{amsthm}
\usepackage[left=2.7cm,right=2.7cm,top=3.5cm,bottom=3.5cm]{geometry}
\usepackage{amsmath,amssymb,amsthm,amscd,amsfonts}
\usepackage{latexsym}
\usepackage{graphicx}
\usepackage[all]{xy}
\usepackage{xcolor}
\usepackage[OT2,T1]{fontenc}
\usepackage{hyperref}
\usepackage{mathtools}
\usepackage{tikz}
\usetikzlibrary{arrows}
\usepackage{float}
\usepackage{caption}
\usepackage[new]{old-arrows}

\usepackage{comment}

\theoremstyle{plain}
\newtheorem{theorem}{Theorem}[section]
\newtheorem{lemma}[theorem]{Lemma}

\newtheorem{corollary}[theorem]{Corollary}
\newtheorem{Problem}[theorem]{Problem}

\theoremstyle{definition}

\theoremstyle{remark}
\newtheorem{remark}[theorem]{Remark}

\DeclareMathAlphabet{\mathpzc}{OT1}{pzc}{m}{it}
\DeclareSymbolFont{cyrletters}{OT2}{wncyr}{m}{n}
\DeclareMathSymbol{\Sha}{\mathalpha}{cyrletters}{"58}

\begin{document}
\thispagestyle{empty}
	
\newcommand{\ZZ}{{\mathbb Z}}
\newcommand{\QQ}{{\mathbb Q}}
\newcommand{\Z}{{\mathbb Z}}
\newcommand{\Q}{\mathbb Q}
\newcommand{\FF}{{\mathbb F}}
\newcommand{\E}{{\mathcal{E}}}
\newcommand{\mcF}{{\mathcal{F}}}
\newcommand{\A}{\mathcal{A}}
\newcommand{\G}{\mathbb{G}}
\newcommand{\chara}{\rm{char}}
\newcommand{\ds}{\displaystyle}
\newcommand{\la}{\langle}
\newcommand{\ra}{\rangle}
\newcommand{\z}{{\zeta}}
\newcommand{\ov}{\overline}
\newcommand{\wt}{\widetilde}
\newcommand{\Or}{\mathcal{O}}
\newcommand{\X}{\mathcal{X}}

\newcommand{\Hil}{\mathcal{H}}
\newcommand{\End}{\rm{End}}
\newcommand{\Aut}{\rm{Aut}}

\newcommand{\loc}{{\rm loc}}
\newcommand{\Gal}{{\rm Gal}}
\newcommand{\Id}{{\rm Id}}
\newcommand{\GL}{{\rm GL}}
\newcommand{\SL}{{\rm SL}}
\newcommand{\Hh}{{\rm H}}
\newcommand{\Mat}{{\rm Mat}}
\newcommand{\disc}{{\rm disc}}
\newcommand{\imm}{{\rm Im}}
\newcommand{\ind}{{\rm Ind}}
\newcommand{\res}{{\rm res}}
\newcommand{\diag}{{\rm diag}}
\newcommand{\NN}{{\rm NN}}
	
\renewcommand{\leq}{\leqslant}
\renewcommand{\geq}{\geqslant}
\newcommand{\modn}{{\rm{mod} \hspace{0.1cm} }}
\newcommand{\nc}{\normalcolor}
\newcommand{\teal}{\color{teal}}
\newcommand{\ja}[1]{{\color{blue}{#1}}}
\newcommand{\lp}[1]{{\color{teal}{#1}}}

\title{On the Hasse principle for divisibility in elliptic curves}
\author{Jessica Alessandr\`i}
\author{Laura Paladino}
\date{\today}

\renewcommand{\thefootnote}{\arabic{footnote}}
\setcounter{footnote}{0}
	
\begin{abstract}
Let $p$ be  a prime number
and $n$ a positive integer. 
Let $\E$ be an elliptic curve defined over a number field $k$. 
 It is known that the local-global divisibility by $p$ holds in $\E/k$,
but for powers of $p^n$ counterexamples may appear. The validity or the failing of the Hasse principle depends
on the elliptic curve $\E$ and the field $k$ and, consequently, on the group  $\Gal(k(\E[p^n])/k)$.
For which kind of these groups does the principle hold? For which of
them can we find a counterexample? The answer to these questions was known for $n=1,2$, but
for $n\geq 3$ they were still open. 
We show some conditions on the generators of $\Gal(k(\E[p^n])/k)$  implying
an affirmative answer to the local-global divisibility by $p^n$ in $\E$ over $k$, for every $n\geq 2$.
We also prove that these conditions are necessary by producing counterexamples in the case when they do not hold. 
 These last results generalize
to every power $p^n$, a result obtained by Ranieri for $n=2$.
\end{abstract}
\maketitle

\medskip

\noindent\textbf{Mathematics~Subject~Classification~(2020)}: Primary 11R34, 11G05; Secondary 14K02, 14G05. 
    
\section{Introduction}
Since 2001 various authors have been concerned with the local-global divisibility problem in commutative algebraic groups posed by Dvornicich and Zannier (see \cite{Dvornicich2001Local-globalGroups, Dvornicich2007OnInteger, Illengo2008CohomologyTorus, Paladino2012OnCurves, Creutz2016OnCurves, Gillibert2018OnVarieties, Alessandri2024LocalglobalTori, Creutz2023TheCurves}
among others) and some related questions (see \cite{Ciperiani2015Weil-ChateletCassels, Creutz2013LocallyDivisible, Dvornicich2022Local-globalGroups}).

\begin{Problem}[Dvornicich and Zannier, \cite{Dvornicich2001Local-globalGroups}]\label{prob:DZ}
	Let $q$ be a fixed positive integer. Let $\mathcal G$ be a commutative algebraic group defined over a number field $k$. Assume that a point $P \in \mathcal G(k)$ has the following property: for all but finitely many places $v$ of $k$ there exists $D_v \in \mathcal G(k_v)$ such that $P = qD_v$. Can we conclude that there exists $D \in \mathcal G(k)$ such that $P = qD$?
\end{Problem}

It suffices to answer the question for every power $p^n$ of prime numbers $p$ to get an answer for every positive integer $q$.
Problem \ref{prob:DZ} was motivated by a particular case of the Hasse-Minkowski Theorem on quadratic forms and by the Grunwald-Wang Theorem, which gives an answer to the problem in split tori of dimension 1 (see \cite{Dvornicich2022Local-globalGroups} for further details). 
A complete answer for algebraic tori of every dimension (even non-split) has been recently given in \cite{Alessandri2024LocalglobalTori}. In the case
of an abelian variety of dimension $g$, some sufficient conditions to get an affirmative answer are presented in \cite{Gillibert2018OnVarieties}. The most  studied case was that of abelian varieties $\E$ of dimension 1.  It is well-known that in elliptic curves defined over number fields the local-global divisibility by $p$ holds (see for instance \cite[Theorem 3.1]{Dvornicich2001Local-globalGroups}). Instead for powers $p^n$, with $n\geq 2$ counterexamples may appear (see \cite{Dvornicich2004, Paladino2012OnGroups, Ranieri2018CounterexamplesCurves} among others).  Counterexamples  over $\QQ$ are known for power $p^n$ of $p=2,3$ \cite{Paladino2012OnGroups, Creutz2016OnCurves}, for every $n\geq 2$, and over $\QQ(\zeta_3)$ are known for powers $3^n$ \cite{Paladino2010}, for every $n\geq 2$ too. Those
give also counterexamples  in all number fields $k$ linearly disjoint from $\QQ$ and respectively $\QQ(\z_3)$ (see Remark \ref{raising}).
Instead the question 
when $p\geq 5$ is not well understood yet. Then from now on we will assume $p\geq 5$. 
Let $K_n:=k(\E[p^n])$ denote the $p^n$-division field of $\E$ over $k$.  In \cite{Paladino2012OnCurves}, the authors
give conditions on the structure of the group $\Gal(K_1/k)$, sufficient to the validity of the local-global principle for $p^n$, with $n\geq 2$,  where $k$ is a number field not containing $\QQ(\z_p+\bar{\z_p})$. This last request on the field is
necessary, as showed in \cite[Section 6]{Paladino2014OnPrinciple}. In \cite{Ranieri2018CounterexamplesCurves} Ranieri investigated more the structure of such a Galois group and showed conditions on all possible $\textrm{Gal}(K_1/k)$ giving counterexamples for the divisibility by $p^2$ (see \cite[Theorem 2]{Ranieri2018CounterexamplesCurves}). Moreover, in \cite[Proposition 7]{Paladino2014OnPrinciple}, it is showed that if there exists a counterexample for $p^n$, then $\textrm{Gal}(K_2/k)$ can be put in triangular form. Nevertheless, when $n\geq 2$ (especially when $n\geq 3$) the problem of finding all possible Galois groups $\textrm{Gal}(K_n/k)$ assuring the validity of the Hasse principle for divisibility by $p^n$, remained in general open.
In this paper, with Theorem \ref{thm:conditions} we answer this question in the open cases, by giving sufficient conditions for the generators of $\textrm{Gal}(K_n/k)$ to have the validity of the local-global principle for divisibility by $p^n$,  for every $n\geq 1$.  Furthermore, in Section \ref{sec4} we show that these hypotheses are necessary by exhibiting counterexamples in the case when they are not satisfied, for every $n\geq 2$ (see Theorem \ref{thm:counterexamples} and Corollary \ref{n+s} and notice that the bound $n\geq 2$ is best possible, since the local-global principle holds for divisibility by $p$ as mentioned above). In this way we generalize to every power $p^n$, with $n\geq 2$, the results produced by Ranieri in \cite{Ranieri2018CounterexamplesCurves} for $p^2$.

\medskip
It is well known that an obstruction to the validity of Problem \ref{prob:DZ} is given by
the first local cohomology group (see \cite[Definition at pag. 321]{Dvornicich2001Local-globalGroups} and Equation \ref{h1loc} in Section \ref{sec2} for the definition of this group, see \cite[Proposition 2.1]{Dvornicich2001Local-globalGroups} and \cite[Theorem 3]{Dvornicich2007OnInteger} for the results about the relationship between its vanishing and the validity of the local-global principle).
Such a group is isomorphic to some modified Tate-Shafarevich group, as we recall in Section \ref{sec2} (see also \cite[\S 3]{Creutz2016OnCurves}, \cite[Proposition 4.1]{Dvornicich2022Local-globalGroups}).
The triviality of this modified Tate-Shafarevich group, along with assuring an affirmative answer to
Problem \ref{prob:DZ}, implies an affirmative answer also to the following second local-global question for divisibility of cohomology classes (see \cite[Theorem 2.1]{Creutz2016OnCurves} and \cite{Dvornicich2022Local-globalGroups}).

\begin{Problem} \label{prob2}  Let $q, t$ be positive integers, let $\sigma\in H^t(k,A)$ and let $res_v: H^t(k,A)\rightarrow H^t(k_v,A)$ be the restriction map.   Assume that for all but finitely many places $v$ of $k$ there exists $\tau_v\in H^t(k_v,A)$ such that $q\tau_v=res_v(\sigma)$. Can we conclude that there exists $\tau \in H^t(k,A)$, such that $q\tau=\sigma$?
\end{Problem}

Our hypotheses on $G_n$ to get an affirmative answer to Problem \ref{prob:DZ} also imply an affirmative answer to Problem \ref{prob2}, as we will see in next section (see also Corollary \ref{cor:prob2}).

\subsection*{Acknowledgements}This work began in February 2024, when the first author was a visiting guest at University of Calabria. She thanks the hosting university for its hospitality and financial support.
Both the authors are grateful to the Italian ``National Group for Algebraic and Geometric Structures, and their Application'' (GNSAGA - INdAM), of which they are members, for partially supporting this work.
The first author was also supported by UKRI Future Leaders Fellowship \texttt{MR/V021362/1} and by the Max-Planck Institute for Mathematics in Bonn, that she thanks for its hospitality and financial support.

\section{Sufficient conditions to the Local-global divisibility by \texorpdfstring{$p^n$}{pn}}
In this section we prove some of the main results of this work. In the first subsection we recall some well known facts about the translation of Problem \ref{prob:DZ} and Problem \ref{prob2} in a cohomological context. In this way, we also depict the relation between the two problems. In the second subsection, we recall what is known on the generators of $\Gal(k(\E[p^n])/k)$ and pick some particular elements of this group whose behaviour is related to the answers of the local-global questions, as we will prove in the rest of the paper. In particular in Subsection \ref{sec3} we show in which cases the answer is affirmative.

\subsection{First local cohomology group and Tate-Shafarevich group}  \label{sec2}
As above, we denote by $p$ a prime number, by $n$ a positive integer, by $k$ a number field and by $\mathcal E$ an elliptic curve defined over $k$. 
Let $K_n:=k(\E[p^n])$ and  $G_n:= \Gal(k(\E[p^n])/k)$, for all $n\geq 1$. 
In addition, we denote by $M_k$ the set of places of $k$, by $k_v$ the completion of $k$ at the place $v$ and by $G_{n,v}$ the
Galois group $\Gal((k(\E[p^n])_w/k_v)$, where $w$ is a place of $k(\E[p^n])$ extending $v$. For every field of characteristic zero $F$, we denote
by $\bar{F}$ a fixed algebraic closure of it and by $G_F$ the absolute Galois group $\Gal(\bar{F}/F)$. 

\medskip
Let $P\in \E(k)$ and  $W\in \E(\bar{k})$ such that $P=p^nW$.  Then we can define a cocycle
$Z=\{Z_{\sigma}\}_{\sigma\in G_n}$ of $G_n$ with values in $\E[p^n]$ by
$$Z_{\sigma}:=\sigma(W)-W, \quad \sigma\in G_n.$$

The hypotheses of Problem \ref{prob:DZ} assure the vanishing of the class of $Z$ in $\Hh^1(G_{n,v},\E[p^n])$, for every $v\in \Sigma$, where
 $\Sigma$ is the subset of $M_k$ containing all the places $v$ of $k$ satisfying the assumptions of the problem, while an affirmative answer would imply its vanishing in $\Hh^1(G_n,\E[p^n])$, see \cite[\S 2]{Dvornicich2001Local-globalGroups}, \cite[Proposition 3.1]{Dvornicich2022Local-globalGroups}.
 It is then natural to consider the group
\begin{equation} \label{h1loc}
\Hh^1_{\textrm{loc}}(G_n,\E[p^n]):=\bigcap_{v\in \Sigma} \ker\{ \Hh^1(G_n,\E[p^n])\xrightarrow{\makebox[1cm]{$\mathrm{res}_v$}} \Hh^1(G_{n,v},\E[p^n])\},
\end{equation}
whose triviality implies an affirmative answer to Problem \ref{prob:DZ}, as proved in \cite[Proposition 2.1]{Dvornicich2001Local-globalGroups}.
If $\Hh^1_{\textrm{loc}}(G_n,\E[p^n])$ is non-trivial, we instead have counterexamples in a finite extension of $k$ \cite[Theorem 3]{Dvornicich2007OnInteger}.
 Recall that $G_{n,v}$ varies among all the cyclic subgroups $\langle \sigma\rangle$ of $G_n$ as $v$ varies in $\Sigma$. 
Then, as stated in \cite[Definition at pag. 321]{Dvornicich2001Local-globalGroups}, the classes $[Z]\in\Hh^1_\loc(G_n, \E[p^n])$ are classes of cocycles $Z=\{Z_\sigma\}_{\sigma \in G_n}$ satisfying the so-called \emph{local conditions}, i.e.\ for every $\sigma\in G_n$,  there exists $W_\sigma \in \E[p^n]$ such that  $Z_\sigma = (\sigma-1)W_\sigma$.

The definition of the group  $\Hh^1_{\textrm{loc}}(G_n,\E[p^n])$ is very similar to that of the Tate-Shafarevich group: 

\[
    \Sha(k,\E[p^n]) :=\bigcap_{v\in M_k} \ker \{ \Hh^1(G_k,\E[p^n]) \xrightarrow{\makebox[1cm]{$\mathrm{res}_v$}} \Hh^1(G_{k_v},\E[p^n])\}.
\]

If in the last definition we let $v$ vary in $\Sigma$ instead of $M_k$, we get 

\[\Sha_{\Sigma}(k,\E[p^n]):=\bigcap_{v\in \Sigma} \ker \{ \Hh^1(G_k,\E[p^n])\xrightarrow{\makebox[1cm]{$\mathrm{res}_v$}} \Hh^1(G_{k_v},\E[p^n])\}.\]

By \cite[Lemma 3.3]{Creutz2012AVarieties} (see also \cite[Proposition 4.1]{Dvornicich2022Local-globalGroups}) we have that $\Hh^1_{\textrm{loc}}(G_n,\E[p^n])$ is isomorphic to $\Sha_{\Sigma}(k,\E[p^n])$. In particular, the triviality of $\Hh^1_{\loc}(G_n,\E[p^n])$ implies $\Sha(k,\E[p^n])=0$.
By \cite[Theorem 2.1]{Creutz2016OnCurves}, the last equality assures and affirmative answer to Problem \ref{prob2}.

\subsection{Sufficient conditions to the local-global divisibility} \label{sec3}
In this section we give sufficient conditions on the structure of $G_n$ to have the validity of the local-global principle. The strategy of the proof is showing that the first local cohomology group, defined in the previous subsection, vanishes under those hypotheses. We are going to describe  some particular elements of $G_n$.\\

 We want to restrict to the cases when an answer to the problem is not known yet. For this purpose we are going to make some assumptions on $k$ and $\E$. \normalcolor 
 We assume that $k$ does not contain $\QQ(\z_{p}+\bar{\z_p})$ (otherwise, one can find counterexamples to the local-global divisibility by $p^n$, as showed in \cite[Section 6]{Paladino2014OnPrinciple}).  If $\E$ has no $k$--rational points of exact order $p$,  Theorem 1 in \cite{Paladino2012OnCurves} assures an affirmative answer to the local-global divisibility by $p^n$. Thus, whenever
 $G_n$ is a group  whose reduction modulo $p$ cannot be put in the form

 $$ \begin{pmatrix}
    1 & \star \\ 0 & \star
\end{pmatrix}$$

 \noindent for every basis of $\E[p^n]$, we have $\Hh^1_{\textrm{loc}}(G_n,\E[p^n])=0$. 
 Then we can assume that $\E$ admits a $k$-rational point of exact order $p$ and  there exists a basis of $\E[p]$ such that every element of $G_1$ can be represented in $\GL_2(\Z/p\Z)$ as a matrix of the form
\[ \begin{pmatrix}
    1 & \star \\ 0 & \chi_p
\end{pmatrix},\]
where $\chi_p$ is the cyclotomic character modulo $p$. In addition,
we can assume that $G_1$ is cyclic of order dividing $p-1$, and it is generated by a matrix of the form
\[ \rho_1 = \begin{pmatrix}
    1 & 0\\
    0 & \lambda_1 
\end{pmatrix}.\]
Otherwise, by combining \cite[Lemma 8]{Paladino2012OnCurves} and \cite[Proposition 6]{Paladino2014OnPrinciple}, we have $\Hh_\loc^1(G_n,\E[p^n])=0$ as well, and the local-global principle for divisibility by $p^n$ holds in $\E$ over $k$. \normalcolor
The element $\lambda_1$ (and so $\rho_1$) has order dividing $p-1$ and greater than or equal to $3$, since $k$ does not contain $\Q(\z_p + \bar{\z_p})$.
 Observe that there exists such a $\lambda_1$, for every $p\geq 5$ (that was our assumption on $p$ from the beginning, see also Remark \ref{raising}).  \normalcolor 
By \cite[Lemma 10]{Paladino2012OnCurves}, we can choose a basis of $\mathcal E[p^n]$ such that $\rho_1$ admits a lift
\[ \rho_n = \begin{pmatrix}
    1 & 0\\
    0 & \lambda_n 
\end{pmatrix},\]
with $\lambda_n \equiv \lambda_1 \mod p$. We fix this basis $\{Q_1, Q_2\}$ for $\E[p^n]$ and we also fix the basis $\{p^{n-i}Q_1, p^{n-i}Q_2\}$ of $\mathcal E[p^i]$, for every $1\leq i\leq n-1$. Observe that $p^{n-1}Q_1$ is a $k$-rational $p$-torsion point of $\E$.
 In addition from now on we assume that $G_2$ is in upper triangular form or in lower triangular form; otherwise, by \cite[Proposition 7]{Paladino2014OnPrinciple} we would have $\Hh_\loc^1(G_n,\E[p^n])=0$. 

Let $\mathcal{D}_n$, $s\mathcal{U}_n$, $s\mathcal{L}_n$ be respectively the group of the diagonal, strictly upper triangular and strictly lower triangular matrices in $G_n$.
By \cite[Proposition 12]{Paladino2012OnCurves} the matrices in $G_n$ decompose as products of elements in these subgroups, then in particular $G_n=\langle \mathcal{D}_n, s\mathcal{U}_n, s\mathcal{L}_n\rangle$. Furthermore, by \cite[Proposition 12]{Paladino2012OnCurves} (see also \cite[Lemma 5]{Paladino2014OnPrinciple}), the groups $s\mathcal{L}_n$ and $s\mathcal{U}_n$ are cyclic and are respectively generated by
$$\tau_L=\begin{pmatrix}
    1 & 0 \\
     p^{j} & 1 \\
\end{pmatrix},$$
where $p^j$ is the smallest power of $p$ dividing the entries $c>1$ of elements of $s\mathcal{L}_n$ and 
by
$$\tau_U=\begin{pmatrix}
    1 & p^i \\
     0 & 1 \\
\end{pmatrix},$$
where $p^i$ is the smallest power of $p$ dividing the entries $b>1$ of elements of $s\mathcal{U}_n$, when $b>1$. Observe
that $j\geq 1$ and $i\geq 1$, by our assumption that $G_1$ is cyclic of order $p-1$, generated by $\rho_1$. By the definitions of $\tau_L$ and
$\tau_U$, we have
$G_n=\langle \mathcal{D}_n, \tau_U, \tau_L \rangle$.

\bigskip

Since $G_1 = \langle \rho_1 \rangle$, every matrix in $\mathcal D_n$ is of the form
\begin{equation} \label{eq:diagonali}
\begin{pmatrix}
    1+a p^t & 0 \\
   0 & \mu \\
\end{pmatrix},
\end{equation}
where  $t\geq 1$ is an integer, $a \in (\Z/p^n\Z)^*$ and $\mu \equiv \lambda_1^k \mod{p}$, for some integer $k$. We choose $m$ to be the minimum of all such integers $t$. Notice that in particular $m\geq 1$. Let $\tilde\delta = \begin{pmatrix}
    1+p^m a & 0 \\
   0 & \mu \\
\end{pmatrix}$ be a matrix associated to $m$. In the proof of \cite[Proposition 7, pag.\ 300]{Paladino2014OnPrinciple} it is showed that, since $a$ is invertible, there exists an integer $l$ such that $(1+p^m a)^l \equiv 1+p^m \mod{p^n}$;
moreover, by taking $(\tilde\delta \rho_n^{-k})^l$, one can find in $G_n$ the following matrix
$$\delta:=(\tilde\delta \rho_n^{-k})^l=\begin{pmatrix}
    1+p^m & 0 \\
   0 & 1+p^hd \\
\end{pmatrix},$$
with $h\geq 1$ an integer and $d \in (\Z/p^n\Z)^*$. 

We are going to observe that with the same argument as in the proof of \cite[Proposition 12]{Paladino2014OnPrinciple} (by swapping the role of $\delta$ and $\tau_L$), one can assume without loss of generality that the class of a cocycle $[Z]=[\{Z_{\sigma}\}_{\sigma \in G_n}]$ in $\mathrm{H}^1_{\mathrm{loc}}(G_n, \mathcal E[p^n])$ has a representative with $Z_\delta=(p^m\beta,0)$, for some $\beta \in \Z/p^n\Z$, and $Z_{\tau_L}=Z_{\tau_U}=Z_{\rho_n}=(0,0)$. 

\begin{lemma} \label{Q}
Let $c \in \Hh^1_{\mathrm{loc}}(G_n, \E[p^n])$. Then there exists a cocycle $Z$ of $G_n$ with values in $\E[p^n]$, such that $[Z] = c$, and
\[ Z_{\tau_U} =(0,0), \quad Z_{\tau_L}=(0,0), \quad Z_{\rho_n}=(0,0),\]
\[ Z_{\delta} = (p^m\beta, 0), \quad \text{for some } \beta \in \Z/p^n\Z. \]
\end{lemma}

\begin{proof}
    As mentioned above, the argument is very similar to the one given in \cite[Proposition 12]{Paladino2014OnPrinciple}. However, for the reader's convenience we state it here in details.   We consider the image of $Z$ through the three restrictions from $G_n$ to $\mathcal D_n$, from $G_n$ to $\langle \rho_n, s\mathcal U_n \rangle$ and from $G_n$ to $\langle \rho_n, s\mathcal L_n \rangle$.
    We still denote with $[Z]$ the images of the class in $\mathrm{H}^1_{\mathrm{loc}}( \mathcal D_n, \mathcal E[p^n])$, in $\mathrm{H}^1_{\mathrm{loc}}(\langle \rho_n, s\mathcal U_n \rangle, \mathcal E[p^n])$ and in $\mathrm{H}^1_{\mathrm{loc}}(\langle \rho_n, s\mathcal L_n \rangle, \mathcal E[p^n])$. By \cite[Proposition 17]{Paladino2012OnCurves}, all of these groups are trivial, thus
    \[ \exists \, Q \in \mathcal E[p^n] \text{ s.t. } \forall \omega \in \mathcal D_n \quad Z_{\omega} = \omega(Q) - Q, \]
    \[ \exists \, P \in \mathcal E[p^n] \text{ s.t. } \forall \gamma \in \langle \rho_n, s\mathcal U_n \rangle \quad Z_{\gamma} = \gamma(P)-P, \]
    \[ \exists \, R \in \mathcal E[p^n] \text{ s.t. } \forall \theta \in \langle \rho_n, s\mathcal L_n \rangle \quad Z_{\theta} = \theta(R)-R. \]
     By adding to $Z$ the coboundary $Z_{\sigma} = \sigma(-R) - (-R)$, we may assume, without  loss of generality, that $R = (0,0)$, i.e.\ $Z_{\theta} = (0,0)$ for every $\theta \in \langle \rho_n, s\mathcal L_n \rangle$. Observe that $\rho_n$ lies in $\mathcal D_n$, $\langle \rho_n, s\mathcal U_n \rangle$ and $\langle \rho_n, s\mathcal L_n \rangle$, so that
    \[ Z_{\rho_n} = \rho_n(Q)-Q = \rho_n(P)-P = \rho_n(R)-R = (0,0). \]
    Therefore, both the point $Q$ and the point $P$ lie in $\ker(\rho_n-1)$. Hence $P=(\alpha,0)$ and $Q=(\beta,0)$, for some
    $\alpha, \beta \in \Z/p^n\Z$.
    With respect to the the matrix $\tau_U$, which is the generator of $s\mathcal U_n$, we have
    \[ Z_{\tau_U} = \tau_U(P)-P = \begin{pmatrix} 0 & p^i\\ 0 & 0 \end{pmatrix} \begin{pmatrix} \alpha \\ 0 \end{pmatrix} = \begin{pmatrix}  0 \\ 0 \end{pmatrix}.
    \]
    On the other hand, with respect to the matrix $\delta$, the image of the cocycle $Z$ is
    \[ Z_\delta = \delta(Q)-Q = \begin{pmatrix} p^m & 0 \\ 0 & p^hd \end{pmatrix} \begin{pmatrix} \beta \\ 0 \end{pmatrix} = \begin{pmatrix} p^m\beta \\ 0 \end{pmatrix}. \]
\end{proof}

\noindent Observe that if $G_n$ is in upper triangular form (with respect to the fixed basis $\{Q_1,Q_2\}$), then as a straightforward consequence of Lemma \ref{Q}, we get that every cocycle $Z$ of $G_n$ with values in $\E[p^n]$ vanishes in $H^1_\loc(G_n,\E[p^n])$. 

\begin{corollary}\label{prop:uptriang}
    If $G_n$ is contained in the group of the upper triangular matrices, then $\Hh^1_{\mathrm{loc}}(G_n, \mathcal E[p^n]) = 0$
    and the local-global divisibility by $p^n$ holds in $\E$ over $k$. 
\end{corollary}

 \noindent   Notice that if $G_n$ is in upper triangular form, then there is a cyclic subgroup of $\mathcal E[p^n]$ of order $p^n$, stable under the Galois action and generated by the first element of the basis, which we chose to be a lifting of a $k$--rational point of exact order $p$. In particular, $\mathcal E$ has a cyclic $k$--rational isogeny of order $p^n$.

\bigskip
 Applying Lemma \ref{Q}, if $m\geq n$, by the minimality of $m$, we have $Z_\sigma=0$, for every $\sigma\in G_n$, implying $\mathrm{H}_{\mathrm{loc}}^1(G_n, \mathcal E[p^n]) = 0$. We have already observed that $m\geq 1$. Hence from now on we assume $1\leq m<n$. In addition, in view of Corollary \ref{prop:uptriang}, we suppose that $G_n$ is not in upper triangular form. 
Notice that in particular we are asssuming that $G_n$ is not in diagonal form (in fact also by \cite[Proposition 11]{Paladino2014OnPrinciple}
we have that  $\mathrm{H}_{\mathrm{loc}}^1(\mathcal{D}_n, \mathcal E[p^n]) = 0$). 

We are going to show  that in many cases we still have an affirmative answer to the problem, even under the assumption that $G_1$ is cyclic generated by $\rho$. Therefore we are going to refine the criterium obtained by combining \cite[Lemma 8]{Paladino2012OnCurves} and \cite[Proposition 6]{Paladino2014OnPrinciple}, by proving the following.

\begin{theorem}\label{thm:conditions}
 With the definitions of $1 \leq i$, $1\leq j<n$, $1\leq m<n$ and $1 \leq h$  as above, if $i \leq h+ |j-m|$ then $\mathrm{H}_{\mathrm{loc}}^1(G_n, \mathcal E[p^n]) = 0$.
\end{theorem}

\begin{proof}
    We treat separately the case when $j \leq m$ and the case when $j>m$.

    \medskip
    \textbf{Case $j \leq m$.}
    Let $[Z]=[\{Z_{\sigma}\}_{\sigma \in G_n}]$ be the class of a cocycle in $\mathrm{H}^1_{\mathrm{loc}}(G_n, \mathcal E[p^n])$. 
    As seen in the proof of Lemma \ref{Q}, we can assume without loss of generality that $Z_{\tau_L}=Z_{\tau_U}=(0,0)$ and that there exists $\beta \in \Z/p^n\Z$, such that $Z_{\omega}=(\omega-1)(\beta,0)$ for every matrix $\omega \in \mathcal D_n$. In particular $Z_\delta=(p^m\beta,0)$.
    Consider the following matrix in $G_n$ 
    
    \[\delta\tau_L^c\tau_U^b=\begin{pmatrix}
     1+p^m & p^ib(1+p^m)\\
     p^jc(1+p^hd) & 1+p^hd + p^{i+j}bc(1+p^hd)\\
     \end{pmatrix},\]

    \noindent with $b,c \in \ZZ$. By the property of being a cocycle, we have 
    $$Z_{\delta\tau_L^c\tau_U^b}= Z_\delta+ \delta (Z_{\tau_L^c\tau_U^b})=Z_\delta=(p^m\beta, 0),$$
    
    \noindent because of $Z_{\tau_L}=Z_{\tau_U}=(0,0)$, which implies $Z_{\tau_L^c\tau_U^b}=(0,0)$. Since $Z$ satisfies the local conditions, there exist $x,y \in \Z/p^n\Z$ such that 
    $$Z_{\delta\tau_L^c\tau_U^b} = (\delta\tau_L^c\tau_U^b-1)(x,y)= (p^m\beta,0).$$ 
    
    \noindent Hence we have the following system of equations:
    \[\begin{cases}
     p^mx + p^ib(1+p^m)y = p^m\beta\\
     p^jc(1+p^hd)x + p^hdy + p^{i+j}bc(1+p^hd)y = 0.\\
    \end{cases}
    \]
    We can choose $\tilde{c} = (1+p^hd)^{-1}$ and so, considering the local conditions with respect to $\delta\tau_L^{\tilde{c}}\tau_U^b$, we get the system
    \[\begin{cases}
     p^mx + p^ib(1+p^m)y = p^m\beta\\
     p^jx + p^hdy + p^{i+j}by = 0.\\
    \end{cases}
    \]
    Because of $i < h+m-j$, we can set $b = p^{h+m-j-i}d$ and rewrite the system as
    \[\begin{cases}
     p^{m-j}(p^jx + p^{h}d(1+p^m)y) = p^m\beta\\
     p^jx + p^hd(1 + p^{m})y = 0.\\
    \end{cases}
    \]
    This implies $p^m\beta =0$.
    Since every matrix $\omega\in \mathcal D_n$ is of the form \eqref{eq:diagonali}, by the minimality of $m$, we have that $Z_{\omega} = (\omega-1)(\beta,0) = (0,0)$.
    By $G_n=\langle \mathcal D_n, s\mathcal L_n, s\mathcal U_n\rangle=\langle \mathcal D_n, \tau_L, \tau_U\rangle$, one can easily deduce that $Z_\sigma=(0,0)$, for every $\sigma\in G_n$. Hence $[Z]$ is a coboundary and $H^1_\loc(G_n,\E[p^n])=0$.

    \medskip
    \textbf{Case $j > m$.}
    In this case, given $[Z] = [\{Z_{\sigma}\}_{\sigma \in G_n}] \in \mathrm{H}^1_{\mathrm{loc}}(G_n, \mathcal E[p^n])$, we may assume without loss of generality as in \cite[Proposition 12]{Paladino2014OnPrinciple}, that $Z_{\delta} = Z_{\tau_U} = (0,0)$ and $Z_{\tau_L} = (0, p^j\beta)$, for some $\beta \in \Z/p^n\Z$. 
    Consider the power $\delta^{p^{j-m}}$, which is equal to
    \[\begin{pmatrix}
        1+p^j a & 0 \\
       0 & 1+p^{h+j-m}e \\
    \end{pmatrix},
    \]
    for some  $a,e\in (\Z/p^n\Z)^*$. By the property of being a cocycle, one sees that 
    $$Z_{\tau_L^{a} \delta^{p^{j-m}} \tau_U^b}=Z_{\tau_L^a}+\tau_{L}^a(Z_{\delta^{p^{j-m}}\tau_U^b})=Z_{\tau_L^a}= (0,p^j\beta a).$$
    Moreover, by considering the local conditions on $Z_{\tau_L^{a}\delta^{p^{j-m}}\tau_U^b}$,
    we have that there exist  solutions $x,y \in \Z/p^n\Z$ of the following system of equations
    \[
        \begin{cases}
            p^{j}ax+p^ib(1+p^ja)y=0 \\
            p^j(1+p^ja)a x+p^{h+j-m}ey+p^{i+j}ab(1+p^ja)y=p^{j}\beta a.\\
        \end{cases}
    \]

\noindent Since $(1+p^ja)$ is invertible, we can factor it out, i.e.\ 
\[
        \begin{cases}
            p^{j}ax+p^ib(1+p^ja)y=0 \\
          (1+p^ja)(p^ja x+p^{h+j-m}(1+p^ja)^{-1}ey+p^{i+j}aby)=p^{j}\beta a.\\
        \end{cases}
    \]
   Observe that if we can choose a particular $b$ such that
   \begin{equation} \label{b} p^ib(1+p^ja)=p^{h+j-m}(1+p^ja)^{-1}e+p^{i+j}ab, \end{equation}
 then we would get that the triviality of the left-hand side of the first equation would imply $p^{j}\beta a=0$
 in the second one. Equation \eqref{b} is equivalent to
 $$ p^{h+j-m}(1+p^ja)^{-1}e= p^ib(1+p^ja)-p^{i+j}ab,$$
 i.e.\ $p^{h+j-m}(1+p^ja)^{-1}e=p^ib$.
 We can choose $b = e(1+p^ja)^{-1}p^{h+j-m -i}$, that satisfies the assumption $h+j-m \geq i$, to get an equality. Therefore $p^j\beta a= 0$.
    This implies $p^j\beta =0$, because of $a$ being an invertible element.
     Therefore $Z_{\tau_L} = (0,0)$.  By $G_n=\langle \mathcal D_n, s\mathcal L_n, s\mathcal U_n\rangle=\langle \mathcal D_n, \tau_L, \tau_U\rangle$, one can easily deduce that $Z_\sigma=(0,0)$, for every $\sigma\in G_n$. So $[Z]$ is a coboundary and $H^1_\loc(G_n,\E[p^n])=0$.
\end{proof}

\noindent As a consequence of Theorem \ref{thm:conditions} we immediately get the following result.

\begin{corollary}
Let $p\geq 5$ be a prime number and $n$ a positive integer. Let $\E$ be an elliptic curve defined over a number field $k$ not containing $\QQ(\z_p+\bar{\z_p})$.  Under the hypotheses of Theorem \ref{thm:conditions} the local-global divisibility by $p^n$ holds in $\E$ over $k$. \end{corollary}

\noindent In the next section we will show that Theorem \ref{thm:conditions} is best possible.
Theorem \ref{thm:conditions} also implies the following statement that refines the criterium given in \cite[Proposition 7]{Paladino2014OnPrinciple}.

\begin{corollary} \label{upper_2}
Let $p\geq 5$ be a prime number and let $n$ be a positive integer.
Let $\E$ be an elliptic curve defined on a number field $k$
not containing $\QQ(\z_{p}+\bar{\z_{p}})$. If $G_1$ is cyclic, generated by
$\begin{pmatrix} 1 & 0\\ 0 & \lambda_1\end{pmatrix}$, with $\lambda_1\in (\ZZ/p\ZZ)^*$, $\rm{ord}(\lambda_1)\geq 3$,
and $G_2$ is in upper triangular form and not diagonal form,  then the local-global
divisibility by $p^n$ holds in $\E$ over $k$, for every positive integer $n$.
\end{corollary}
 
\begin{proof} If $G_2$ is in upper triangular form and not diagonal, then $i = 1 \leq h$. The conclusion follows immediately by Theorem \ref{thm:conditions}.
\end{proof}

\noindent Observe that the conclusion of Corollary \ref{upper_2} is obtained after having fixed a basis $\{Q_1, Q_2\}$ of $\E[p^n]$
such that $p^{n-1} Q_1$ is a $k$-rational point. In fact, if $G_2$ is in lower triangular form with respect to such a basis,
then we can have counterexamples as we shall see in the following section.  If we swap the role of $Q_1$ and $Q_2$ and choose
$Q_2$ such that $p^{n-1} Q_2$ is a $k$-rational point, then the conclusion of Corollary \ref{upper_2} holds instead when
$G_2$ is in lower triangular form and we can have counterexamples when it is in upper triangular form as in
\cite[Lemma 10]{Ranieri2018CounterexamplesCurves}.

As mentioned above, in view of
\cite[Proposition 7]{Paladino2014OnPrinciple} and taking into account \cite[Lemma 8]{Paladino2012OnCurves}
and \cite[Proposition 6]{Paladino2014OnPrinciple}, by Corollary  \ref{upper_2}  
we can give the following criterium which reduces further the possible cases when counterexamples may appear.

\begin{corollary} \label{criterium}
Let $p\geq 5$ be a prime number and let $n$ be a positive integer.
Let $\E$ be an elliptic curve defined on a number field $k$
not containing $\QQ(\z_{p}+\bar{\z_{p}})$. If $\Hh^1_\loc(G_n,\E[p^n])\neq 0$ then there exists a basis
of $\E[p^n]$ such that $G_1$ is cyclic, generated by
$\begin{pmatrix} 1 & 0\\ 0 & \lambda_1\end{pmatrix}$, with $\lambda_1\in (\ZZ/p\ZZ)^*$, $\rm{ord}(\lambda_1)\geq 3$,
and $G_2$ is in lower triangular form. 
In particular $\E$ admits a $k$-rational point $Q$ of order $p$ and
a $k$-rational isogeny of degree $p^2$, whose kernel does not contain contains $Q$.
\end{corollary}

\noindent In the proof of Theorem \ref{thm:counterexamples} we will show that 
the hypotheses of Corollary \ref{criterium} cannot be improved further, since
there exist counterexamples where all the matrices in $G_2$ are in diagonal form (see Remark \ref{rem_criterium}).

\bigskip
The proof of Theorem \ref{thm:conditions} implies that under the same hypotheses we have the vanishing
of $\Sha(k,\E[p^n])$ and then an affirmative answer to Problem \ref{prob2}, as recalled in Section \ref{sec2}.

\begin{corollary} \label{cor:prob2}
With the definitions of $1 \leq i$, $1\leq j<n$, $1\leq m<n$ and $1 \leq h$  as above, if $i \leq h+ |j-m|$ then $\Sha(k, \mathcal E[p^n]) = 0$ and the local-global divisibility by $p^n$
holds in $H^t(k,\E[p^n])$, for every positive integer $t$. 
\end{corollary}

\section{Counterexamples} \label{sec4}

\noindent This section is devoted to the proof that the conditions given in Theorem \ref{thm:conditions} are
necessary. We produce counterexamples in the cases when they are not satisfied. By Theorem \ref{thm:conditions},
possible counterexamples may appear only in these situations:
\[ \begin{cases}
    j < m\\
    i > h + m -j,
\end{cases}
\quad
\text{ or }
\quad 
\begin{cases}
    j \geq m\\
   i > h + j -m,
\end{cases}
\]
 with $1 \leq i$, $1\leq j<n$, $1\leq m<n$ and $1 \leq h$. 
 
\begin{remark} \label{raising}
 As mentioned in the Introduction, it is known that for powers $p^n$, with $p\in \{2,3\}$ and $n\geq 2$, there exist counterexamples over $\QQ$ \cite{Paladino2012OnGroups, Creutz2016OnCurves}. 
 Moreover, for powers $3^n$, with $n\geq 2$, there exist counterexamples over $\QQ(\z_3)$ \cite{Paladino2010}.  
 All of  these give also counterexamples in every
 extension $L$ of $k$, linearly disjoint from $k(\E[p^n])$,  where $k=\QQ$ or respectively $\QQ(\E[3])$, because of $H^1_\loc(k(\E[p^n])/k)\simeq H^1_\loc(L(\E[p^n])/L)$.
 Therefore we search for counterexamples for $p\geq 5$. 
 \end{remark}

\begin{theorem}\label{thm:counterexamples}
Let $1 \leq i$, $1\leq j<n$, $1\leq m<n$ and $1 \leq h$ be defined as in Section \ref{sec3}.
For every prime number $p \geq 5$ and every positive integer  $n \geq 2$ both the following hold

\begin{enumerate}
 \item[1)] there exist groups $G_n=\langle \tau_L, \tau_U, \rho, \delta\rangle$ as above, such that $j< m$, $i> h+ m-j$ and $\Hh^1_\loc(G_n,(\ZZ/p^n\ZZ)^2)\neq 0$;
\item[2)] there exist groups $G_n=\langle \tau_L, \tau_U, \rho, \delta\rangle$ as above, such that $j\geq m$, $i> h+ j-m$ and $\Hh^1_\loc(G_n,(\ZZ/p^n\ZZ)^2)\neq 0$.
\end{enumerate}

\noindent Moreover for every $G_n$ as in $\mathrm{1)}$ and in $\mathrm{2)}$ there exists an elliptic curve
$\E$ defined over a number field $k$ such that $\Gal(k(\E[p^n])/k) \simeq G_n$
and $\Hh^1_\loc(G_n,\E[p^n])\neq 0$.
\end{theorem}
\bigskip
We divide the proof of Theorem \ref{thm:counterexamples} in two parts: the case when $j<m$ and the case when
$j\geq m$.

\begin{proof}[Proof of Theorem \ref{thm:counterexamples} for $j<m$.]

Let $n=2$. 
Recall  we are assuming that $G_n=G_2$ is in lower triangular form and not in diagonal form. 
  Then for $n=2$ the case when $j<m$ does not hold under our assumptions that $j>0$ and $m<n$.

Hence we can suppose $n \geq 3$. \normalcolor
We consider a group $G_n$ generated by the following automorphisms:
$$\tau_L=\begin{pmatrix}
1 & 0\\
p^{n-2} & 1\\
\end{pmatrix},
\quad
\tau_U=\begin{pmatrix}
1 & p^{i}\\
0 & 1\\
\end{pmatrix},
\quad
\delta=\begin{pmatrix}
1 +p^{n-1}& 0\\
0 & 1+p^{h}\\
\end{pmatrix},
\quad
\rho=\begin{pmatrix}
1 & 0\\
0 & \lambda\\
\end{pmatrix},$$
with $\lambda=\alpha+p^{h+1}\theta$, for some $\alpha \in (\Z/p\Z)^{*}$ with $\textrm{ord}(\alpha)\geq 3$, $\theta \in \Z/p^n\Z$ and $i>h+1$. We have $m=n-1$ and $j=n-2$. Moreover we set 
\[h = \begin{cases} \dfrac{n}{2} & \text{if }n \text{ even}\\ \dfrac{n-1}{2} & \text{if } n \text{ odd}.\end{cases}\]
Notice that the assumption $i>h+1$ implies $i>h+m-j$. Observe that $\tau_U$ and $\tau_L$ commute since $i+j\geq n$ and that $\tau_U$ and $\delta$ commute because of $i+m\geq n$ and $i+h\geq n$. One can easily verify that $\langle \tau_L, \tau_U \rangle$ is normal in $G_n=\langle \tau_L, \tau_U, \delta, \rho \rangle$. Moreover $\langle \tau_L, \tau_U, \delta \rangle$ is also normal in $G_n$, since it is the kernel of the reduction  modulo $p$ from $G_n$ to $\GL_2(\Z/p\Z)$.
Therefore we have the following chain of normal subgroups
\[\la \tau_L\ra \trianglelefteq \la \tau_L,\tau_U \ra \trianglelefteq \la \tau_L, \tau_U, \delta \ra \trianglelefteq \la \tau_L, \tau_U, \delta, \rho \ra = G_n. \]
Thus every matrix $\sigma\in G_n$ can be written as a product $\delta^a\tau_L^c\tau_U^b\rho^\gamma$, for some integers $a,b,c,\gamma$. Observe that 
$(1+p^{n-1})^a\equiv 1+ap^{n-1} \mod p^n$ and $(1+p^{h})^a\equiv 1+ap^h+\binom{a}{2}p^{2h}\mod p^n$, because of our choices of $m=n-1$ and $h\geq (n-1)/2$ (in particular we have $2(n-1)\geq n$ and $3h\geq n$). Then
$$\sigma\equiv\begin{pmatrix}
(1+p^{n-1})^a & \lambda^\gamma b p^i\\
cp^{n-2}(1+p^h)^a & \lambda^\gamma(1+p^{h})^a\\
\end{pmatrix} \equiv 
\begin{pmatrix}
1+ap^{n-1} & \lambda^\gamma b p^i\\
cp^{n-2}(1+ap^h) & \lambda^\gamma\left(1+ap^{h} + \binom{a}{2}p^{2h}\right)\\
\end{pmatrix} \modn p^n.$$
Let $Z=\{Z\}_{\sigma\in G_n}$ be defined by $Z_\sigma=\begin{pmatrix} ap^{n-1}\\ 0\end{pmatrix}$. 
We are going to verify that $Z$ is a cocycle of $G_n$ with values in $(\Z/p^n\Z)^2$. In the following we will also denote the matrix $\delta^a\tau_L^c\tau_U^b\rho^\gamma$ by $\sigma(a,b,c,\gamma)$.
Given $\sigma_1 = \sigma(a_1,b_1,c_1,\gamma_1) = \delta^{a_1}\tau_L^{c_1}\tau_U^{b_1}\rho^{\gamma_1}$ and $\sigma_2 = \sigma(a_2,b_2,c_2,\gamma_2) = \delta^{a_2}\tau_L^{c_2}\tau_U^{b_2}\rho^{\gamma_2}$, we look at the product
$\sigma_1\sigma_2$. We have $(1+a_1p^{n-1})a_2p^{n-1} \equiv a_2p^{n-1} \mod p^n$ and, for our choices of $h$ and $i$, we also have $c_1p^{n-2}(1+a_1p^h)a_2p^i \equiv 0 \mod p^n$. Then
\[  \sigma_1\sigma_2 \equiv \begin{pmatrix}
    1 + (a_1+a_2)p^{n-1} & (b_2 + b_1\lambda^{\gamma_1})\lambda^{\gamma_2}p^i\\
    (c_1 + \lambda^{\gamma_1}c_2(1+a_2p^h))(1+a_1p^h)p^{n-2} & \lambda^{\gamma_1 + \gamma_2}(1+p^h)^{a_1+a_2}
\end{pmatrix} \mod p^n, \]

\noindent i.e. 

\[ \begin{split} \sigma_1\sigma_2 & \equiv   
\begin{pmatrix}
    1 + (a_1+a_2)p^{n-1} & (b_1+\lambda^{-\gamma_2}b_2)\lambda^{\gamma_1+\gamma_2}p^i\\
    (c_1(1+p^h)^{-a_2}+\lambda^{\gamma_1}c_2)(1+(a_1+a_2)p^h)p^{n-2} & \lambda^{\gamma_1 + \gamma_2}(1+p^h)^{a_1+a_2}
\end{pmatrix} \mod p^n. \\
\end{split}\]

\noindent Observe that $ (1+p^h)^{a_1+a_2}p^{n-2} \equiv (1+(a_1+a_2)p^h)p^{n-2}\mod p^n$, hence

$$\sigma_1\sigma_2=\sigma(a_1+a_2, b_1 + \lambda^{-\gamma_2}b_2, c_1(1+p^h)^{-a_2}+\lambda^{\gamma_1}c_2, \gamma_1+\gamma_2)$$ 

\noindent and 
$$ Z_{\sigma_1\sigma_2} = \begin{pmatrix} (a_1+a_2)p^{n-1} \\0 \end{pmatrix}.$$

\noindent On the other hand

$$Z_{\sigma_1} + \sigma_1 Z_{\sigma_2} = \begin{pmatrix} a_1 p^{n-1}\\ 0 \end{pmatrix} +
\begin{pmatrix}
(1+p^{n-1})^{a_1} & \lambda^\gamma b_1 p^i\\
c_1p^{n-2}(1+p^h)^a & \lambda^\gamma(1+p^{h})^{a_1}\\
\end{pmatrix}
\begin{pmatrix} a_2 p^{n-1}\\ 0 \end{pmatrix}
=\begin{pmatrix} (a_1+a_2) p^{n-1}\\ 0 \end{pmatrix}$$

\medskip
\noindent and thus $Z$ defines a cocycle of $G_n$ with values in $(\Z/p^n\Z)^2$. 
The class of $Z$ in $\Hh^1(G_n, (\Z/p^n\Z)^2)$ belongs to $\Hh^1_\loc(G_n,(\ZZ/p^n\ZZ)^2)$ if and only if it satisfies the local conditions, i.e.\ if and only if the following system has a solution $(x,y)\in (\ZZ/p^n\ZZ)^2$, for all integers $a,b,c,\gamma$:
\[\begin{cases}
    ap^{n-1} x + \lambda^\gamma b p^i y \equiv a p^{n-1} \mod p^n\\
    c p^{n-2}(1+ap^h)x + (\lambda^\gamma(1 +p^h)^a-1)y \equiv 0 \mod p^n.
\end{cases}\]
If $p\mid a$, then $ap^{n-1}=0$ and $(x,y)=(0,0)$ is a solution.
Hence we assume $p\nmid a$.
If $p\nmid \lambda^\gamma(1 +p^h)^a-1$, then
\[
\begin{cases}
    x = 1\\
    y = -c(1+ap^h)(\lambda^\gamma(1 +p^h)^a-1)^{-1} p^{n-2}
\end{cases}
\]
is a solution of the system.
Suppose $p\mid \lambda^\gamma(1 +p^h)^a-1$. Observe that in particular $\lambda^\gamma \equiv 1 \mod p$.
Moreover, we have that $\lambda^\gamma=(\alpha+p^{h+1}\theta)^\gamma=\alpha^\gamma+\sum_{t=1}^\gamma \binom{\gamma}{t}\alpha^{\gamma-t}p^{t(h+1)}\theta^{t(h+1)}=1+p^{h+1}\omega$, for some $\omega\in \ZZ/p^n\ZZ$.

Thus $\lambda^\gamma(1 +p^h)^a-1=\lambda^\gamma-1 +\lambda^\gamma\left(ap^h + \binom{a}{2}p^{2h}\right)= p^h\left(a+p\omega + \binom{a}{2}p^h\right).$ Notice that $a+p\omega + \binom{a}{2}p^h$ is invertible, because of our assumption $p\nmid a$. A solution is then 
\[
\begin{cases}
    x  = 1\\
    y  = -c(1+ap^h)(a+p\omega + \binom{a}{2}p^h)^{-1}p^{n-2-h}
\end{cases}
\]
(recall that $n\geq 3$ and $n-2-h\geq 0$, by our choice of $h$). The cohomology class $[Z]$ is not a coboundary, since the solution of the system depends on $a$, $c$ and $\gamma$.
One can verify this directly: for $\sigma=\delta$ one of the equations of the system is $p^{n-1}x\equiv p^{n-1} \mod p^n$, whose solutions are $x\equiv 1 \mod p$; but for $\tau_L$ we get the equation $p^{n-2}x\equiv 0 \mod p^n$, whose solutions are instead $x\equiv 0 \mod p^2$. 
\par By  \cite[Lemma 11]{Gillibert2017OnVarieties}, given $n \geq 3$ a positive integer and $p \geq 5$ a prime number, there exists a number field $k$ and an elliptic curve $\mathcal E$ over $k$ such that $\Gal\left( k\left(\mathcal{E}[p^n]\right)/k \right)$ is isomorphic to $G_n$ 
defined above. Then in particular $\Hh^1_\loc (G_n,\E[p^n])\neq 0$. 
\end{proof}

\normalcolor
\begin{remark}
Observe that these examples work even if $i=n$ and $\tau_U$ is the identity matrix. In this last
case the group $G_n$ giving the counterexample is in lower triangular form (again having fixed a basis
$\{Q_1,Q_2\}$ of $\E[p^n]$, with $p^{n-1}Q_1$ a $k$-rational point, from the beginning). In particular, this
happens  when $n=4$, where the condition $i> h+m-j = 2+3-2 = 3$ implies $i\geq 4$ and when $n=3$,
where the condition $i>h+m-j=1+2-1=2$ implies $i\geq 3$. 
\end{remark}

\begin{remark}\label{rmk:lambda}
  One can produce other counterexamples by choosing $\lambda=\alpha+p^{h+s}\theta$, with $1\leq s < n/2$, when $n$ is even, and $1\leq s < (n+1)/2$, when $n$ is odd. The same argument in the above proof of Theorem \ref{thm:counterexamples} for $j<m$ work with these other choices of $\lambda$ as well.
\end{remark}

\begin{proof}[Proof of Theorem \ref{thm:counterexamples} in the case when $j\geq m$]

We assume first that $n\geq 4$.
We consider a group $G_n$ generated by the following automorphisms:
\[\tau_L=\begin{pmatrix}
1 & 0\\
p^{n-1} & 1\\
\end{pmatrix},
\quad
\tau_U=\begin{pmatrix}
1 & p^i\\
0 & 1\\
\end{pmatrix},
\quad
\delta=\begin{pmatrix}
1 +p^{m}& 0\\
0 & 1+p^h\\
\end{pmatrix}
\quad
\rho=\begin{pmatrix}
1 & 0\\
0 & \lambda\\
\end{pmatrix},\]
where $\lambda=\alpha+p^{h+2}\theta$, for some $\alpha\in (\ZZ/p\ZZ)^*$ such that $\textrm{ord}(\alpha)\geq 3$, $\theta \in \Z/p^n\Z$, and $i>h+1$.
We are going to show that $\Hh^1_\loc(G_n,(\ZZ/p^n\ZZ)^2)\neq 0$ for such a group $G_n$ with $m=n-1$, for the case when $j=m$, and
with $m=n-2$, for the case when $j>m$. Therefore, from now on we set $m$ in this way and we give a unique proof for both
 these cases. Moreover, we set 
\[h = \begin{cases} \dfrac{n}{2} & \text{if }n \text{ even}\\ \dfrac{n+1}{2} & \text{if }n \text{ odd}.\end{cases}\]
We have that $i$ satisfies $i>h+j-m$, and that $\tau_U$, $\tau_L$ and $\delta$ commute, since $i+j\geq n$, $m+j\geq n$, $j+h\geq n$, $i+m \geq n$ and $i+h\geq n$. As in the previous case, the subgroup $\langle \tau_U, \tau_L, \delta \rangle$ is normal in $G_n$. Thus every matrix $\sigma\in G_n$ can be written as a product $\delta^a\tau_L^c\tau_U^b\rho^\gamma$ for some integers $a,b,c,\gamma$. By $2h\geq n$ and $2m\geq n$, we have
\[\sigma=\begin{pmatrix}
1+ap^{m} & \lambda^\gamma b p^i\\
cp^{n-1} & \lambda^\gamma(1+ap^{h})\\
\end{pmatrix}.\]
Let $Z=\{Z\}_{\sigma\in G_n}$ be defined by $Z_\sigma=\begin{pmatrix} ap^{m}\\ 0\end{pmatrix}$.
We are going to show that $Z$ is a cocycle of $G_n$ with values in $(\Z/p^n\Z)^2$.
Given $\sigma_1 =  \delta^{a_1}\tau_L^{c_1}\tau_U^{b_1}\rho^{\gamma_1}$ and $\sigma_2 = \delta^{a_2}\tau_L^{c_2}\tau_U^{b_2}\rho^{\gamma_2}$, 

we have
\[  \sigma_1\sigma_2 = \begin{pmatrix}
    1 + (a_1+a_2)p^{m} & (b_2 + b_1\lambda^{\gamma_1})\lambda^{\gamma_2}p^i\\
    (c_1 + \lambda^{\gamma_1}c_2)p^{n-1} & \lambda^{\gamma_1 + \gamma_2}\left(1+(a_1+a_2)p^h\right)
\end{pmatrix}. \]
Then $Z_{\sigma_1\sigma_2} = ((a_1+a_2)p^{m},0)$ and, by $(1+a_1p^{m})a_2p^{m} \equiv a_2p^{m} \mod p^n$ and $c_1p^{n-1}a_2p^{m} \equiv 0 \mod p^n$, we get
\[Z_{\sigma_1} + \sigma_1 Z_{\sigma_2} = \begin{pmatrix} a_1 p^{m}\\ 0 \end{pmatrix} + \begin{pmatrix} \left(1+a_1p^{m}\right)a_2p^{m} \\ c_1p^{n-1}a_2p^{m} \end{pmatrix}
= \begin{pmatrix} \left(a_1+a_2\right)p^{m} \\ 0 \end{pmatrix} =  Z_{\sigma_1\sigma_2} .\]
Therefore $Z$ represents a class of a cocycle in $\Hh^1(G_n, (\Z/p^n\Z)^2)$. To have that $[Z]$ actually lies in $\Hh_{\loc}^1(G_n, (\Z/p^n\Z)^2)$, we need to check that $Z$ satisfies the local conditions.
This holds if and only if the following system has a solution $(x,y)\in (\ZZ/p^n\ZZ)^2$, for all integers $a,b,c,\gamma$:
\[\begin{cases}
    ap^{m} x + \lambda^\gamma b p^i y \equiv a p^{m} \mod p^n\\
    c p^{n-1} x + (\lambda^\gamma-1 +a\lambda^\gamma p^{h}) y \equiv 0 \mod p^n.
\end{cases}\]
If $p^2\mid a$, then $ap^{m}=0$ and a solution is $(x,y)=(0,0)$. Hence we can assume $p^2\nmid a$.
If $\lambda^\gamma-1 +a\lambda^\gamma p^h$ is invertible, then a solution is 
\begin{equation}
\begin{cases}
    x = 1\\
    y = -c(\lambda^\gamma-1 +a\lambda^\gamma p^h)^{-1} p^{n-1}.
\end{cases}
\end{equation}
We now assume that $p\mid \lambda^\gamma-1 +a\lambda^\gamma p^h$. As in the case when $j<m$, we have that $\lambda^\gamma \equiv 1 \mod p$ and  
$\lambda^{\gamma} = 1 + p^{h+2}\omega$, for some $\omega \in \Z/p^n\Z$. 
Thus $\lambda^\gamma-1 +a\lambda^\gamma p^h=p^h(a+p^2\omega)$.
If $p\nmid a$, then $a+p^{2}\omega$ is invertible and a solution is 
$$\begin{cases}
    x = 1\\
    y = -c p^{n-h-1}(a+p^{2}\omega)^{-1}
\end{cases}$$ 
\noindent (recall that $n \geq 4$, so $n-h-1 \geq 0$). If $p \mid a$ and $m = j = n-1$, then we are again in the case when $ap^m = 0$ and a solution is $(x,y)= (0,0)$.
Thus suppose that $p \mid a$ and $m = n-2$.
Since we are assuming that $p^2 \nmid a$, we can write $a=p\eta$, with $\eta\in (\ZZ/p\ZZ)^*$. Thus 
\[\lambda^\gamma-1 +a\lambda^\gamma p^h=p^{h+2}\omega +\eta p^{h+1}(1+p^{h+2}\omega) \equiv p^{h+1}(\eta+p\omega) \mod p^n,\]
with $\eta + p\omega$ invertible. A solution is then 
$$\begin{cases}
    x=1\\
    y = -c p^{n-h-2}(\eta+p\omega)^{-1}
\end{cases}$$

\noindent (again, we are assuming $n \geq 4$, so $n-h-2\geq 0$). It remains to show that this cocycle is not a coboundary. This is immediate, since the solution of the system depends on the integers $a$, $c$ and $\gamma$. However, one can verify this directly: for $\sigma=\delta$ one of the equations of the system given by the local conditions is $p^{m}x\equiv p^{m} \mod p^n$, whose solutions are $x\equiv 1 \mod p^2$ if $m = n-2$ or $x \equiv 1 \mod p^{n-1}$ if $m = j = n-1$. On the other hand, for $\tau_L$ we get the equation $p^{n-1}x\equiv 0 \mod p^n$, whose solutions are instead $x\equiv 0 \mod p$.

\bigskip  We now study the case where $n=3$. 
We assume first that $m=j$. Consider a group $G_3$ generated by the following automorphisms:
\[\tau_L=\begin{pmatrix}
1 & 0\\
p^{2} & 1\\
\end{pmatrix},
\quad
\delta=\begin{pmatrix}
1 +p^2& 0\\
0 & 1+p\\
\end{pmatrix},
\quad
\rho=\begin{pmatrix}
1 & 0\\
0 & \lambda\\
\end{pmatrix},\]
where $\lambda=\alpha+p^2\theta$, for some $\alpha\in (\ZZ/p\ZZ)^*$, with $\textrm{ord}(\alpha)\geq 3$ and $\theta \in \Z/p^3\Z$.
 Recall that we assumed that $i$ is a positive integer. Then here we are setting $i=n$, which satisfies $i>h+j-m$, as required. 
We are going to show that $\Hh^1_\loc(G_3,\E[p^3])\neq 0$, for such a group $G_3$. One can easily verify that
$$\la \tau_L\ra \trianglelefteq \la \tau_L, \delta \ra \trianglelefteq \la \tau_L, \delta, \rho\ra = G_3 $$
(observe also that $\la \tau_L, \delta \ra$ is an abelian group in this case).
Thus every matrix $\sigma\in G_3$ can be written as a product 
$\delta^a\tau_L^c\rho^\gamma$ for some integers $a,c,\gamma$, i.e.\
\[\sigma=\delta^a\tau_L^c\rho^\gamma=\begin{pmatrix}
1+ap^2 & 0\\
cp^2 & \lambda^\gamma(1+ap + \binom{a}{2}p^2)\\
\end{pmatrix},\]

\noindent Let $Z=\{Z_\sigma\}_{\sigma\in G_3}$, with $Z_\sigma=\begin{pmatrix} (1+p^2)^a-1 \\ 0\end{pmatrix} = \begin{pmatrix} ap^2 \\ 0\end{pmatrix} $.
We are going to verify that this defines a cocycle of $G_3$ with values in $\Z/p^3\Z$.
Given $\sigma_1 = \delta^{a_1}\tau_L^{c_1}\rho^{\gamma_1}$ and $\sigma_2 = 
\delta^{a_2}\tau_L^{c_2}\rho^{\gamma_2}$,  we have
\[  \sigma_1\sigma_2 = \begin{pmatrix}
    (1+p^2)^{a_1+a_2} & 0\\
    (c_1 + \lambda^{\gamma_1}c_2)p^2 & \lambda^{\gamma_1 + \gamma_2}\left(1+p\right)^{a_1+a_2}
\end{pmatrix}. \]
The image of $Z$ on $\sigma_1\sigma_2$ is 
$$Z_{\sigma_1\sigma_2} = \begin{pmatrix}
  (1+p^2)^{a_1+a_2}-1 \\
  0\\
\end{pmatrix} = \begin{pmatrix}
(a_1+a_2)p^2\\
0\\
\end{pmatrix}$$ 

\noindent and
\[Z_{\sigma_1} + \sigma_1 Z_{\sigma_2} = 
\begin{pmatrix} a_1 p^{2} \\ 0 \end{pmatrix} + \begin{pmatrix} a_2p^2\\ 0 \end{pmatrix}.
\]

\noindent Therefore $Z$ represents the class of a cocycle in $\Hh^1(G_3, (\Z/p^3\Z)^2)$.
We are going to show that $Z$ satisfies the local conditions, i.e.\ that the equation
$$(\sigma-\Id)\begin{pmatrix}x \\ y \end{pmatrix}=\begin{pmatrix} (1+p^2)^a-1 \\ 0 \end{pmatrix}$$
admits a solution, for all $a,c, \gamma$. This yields to the following system of equations
\[\begin{cases}
    ap^2 x\equiv ap^2 \mod p^3\\
    c p^2 x + \left(\lambda^\gamma-1 +a\lambda^\gamma p+\binom{a}{2}\lambda^\gamma p^2\right) y \equiv 0 \mod p^3 .
\end{cases}\]
If $p\mid a$, then $(x,y)=(0,0)$ is a solution of the system. So assume that $p\nmid a$.
If $\lambda^\gamma-1 +a\lambda^\gamma p+\binom{a}{2}\lambda^\gamma p^2$ is invertible, then a solution is given by
\[
\begin{cases}
    x=1\\
    y = -c(\lambda^\gamma-1 +a\lambda^\gamma p+\binom{a}{2}\lambda^\gamma p^2)^{-1}  p^2.
\end{cases}
\]
Suppose that $p|\lambda^\gamma-1 +a\lambda^\gamma p+\binom{a}{2}\lambda^\gamma p^2$. Observe that $\lambda^\gamma= \alpha^\gamma+p^2 \eta \equiv 1+p^2\eta,$ for some $\eta\in \ZZ/p^3\ZZ$ and
$\lambda^\gamma-1 + a\lambda^\gamma p+\binom{a}{2}\lambda^\gamma p^2\equiv a p + \omega p^2=p(a+\omega p)$, for some $\omega \in \Z/p^3\Z$.
We are assuming that $p\nmid a$, so that $a+\omega p$ is invertible and a solution of the system is given by 
\[
\begin{cases}
    x= 1\\
    y = -c(a+\omega p)^{-1}p.
\end{cases}
\]
\noindent Since the solution of the system depends on $a, c$ and $\gamma$,  it is clear that $Z$ is not a coboundary. Anyway one can verify this directly: for $\sigma=\delta$ the first equation in the system is $p^2 x \equiv p^2 \mod p^3$, whose solutions are $x\equiv 1 \mod p$. On the other hand, for $\tau_L$ we get that the second equation in the system is $p^2 x \equiv 0 \mod p^3$, whose solutions are instead $x\equiv 0 \mod p$.

\medskip

\bigskip Assume that $n=3$ and $j>m$. Recall that
$(\ZZ/p^3\ZZ)^* \simeq \ZZ/{p^2}\ZZ\times \ZZ/(p-1)\ZZ\simeq \ZZ/{p^2(p-1)}\ZZ$.
Then we can choose $\lambda\in (\ZZ/p^3\ZZ)^*$ such that  $\textrm{ord}(\lambda)=p-1$.
We consider the group $G_3$ generated by the following automorphisms:
\[\tau_L=\begin{pmatrix}
1 & 0\\
p^2 & 1\\
\end{pmatrix},
\quad
\delta=\begin{pmatrix}
1 +p& 0\\
0 & 1+p\\
\end{pmatrix},
\quad
\rho=\begin{pmatrix}
1 & 0\\
0 & \lambda\\
\end{pmatrix},\]

\noindent
where $\lambda$ is the element of order $p-1$ as above. Notice that $1+p$ has instead order $p^2$ in $(\ZZ/p^3\ZZ)^*$.
In particular, for all positive integers $a$ and $\gamma$, we have $(1+p)^a\in \langle 1+p\rangle \simeq \ZZ/{p^2}\ZZ$ and $\lambda^\gamma\in \langle \lambda\rangle\simeq \ZZ/(p-1)\ZZ$
and in particular $(1+p)^a$ and $\lambda^\gamma$ are not inverse to each other, unless
$(1+p)^a\equiv \lambda^\gamma\equiv 1 \mod p^3$. In addition, observe that $\delta$ is a scalar matrix, but $\delta-1$ does not represent an automorphism of $\E[p^3]$, because of $\det(\delta-1)=p^2$. Then the hypotheses of \cite[Chap.\ V, Theorem 5.1]{Lang1978EllipticCurves} are not satisfied and we can have $\Hh^1(G_3,\E[p^3])\neq 0$. Indeed we are 
 going to show that  the latter holds. 
 As in the case when $n=3$ and $j=m$, here we are taking
 $i=3$. One can verify that there is the following chain of normal subgroups
$$\la \tau_L\ra \trianglelefteq \la \tau_L, \delta \ra \trianglelefteq \la \tau_L, \delta, \rho\ra = G_3$$
(observe that $\delta$ commutes with every other element in $G_3$) and then every matrix $\sigma\in G_3$ can be written as a product $\delta^a\tau_L^c\rho^\gamma$ for some integers $a,c,\gamma$. Thus\
\[\sigma=\delta^a\tau_L^c\rho^\gamma=\begin{pmatrix}
(1+p)^a & 0\\
cp^2 & \lambda^\gamma (1+p)^a\\
\end{pmatrix}.\]

\noindent Let $Z=\{Z_\sigma\}_{\sigma\in G_3}$, with $Z_\sigma=\begin{pmatrix} 0 \\ c p^2\end{pmatrix}$.
We are going to show that this defines a cocycle of $G_3$ with values in $\Z/p^3\Z$.
Let $\sigma_1 = \delta^{a_1}\tau_L^{c_1}\rho^{\gamma_1}$ and $\sigma_2 =  \delta^{a_2}\tau_L^{c_2}\rho^{\gamma_2}$. Hence
\[  \sigma_1\sigma_2 = \begin{pmatrix}
    (1+p)^{a_1+a_2} & 0\\
    (c_1 + \lambda^{\gamma_1}c_2)p^2 & \lambda^{\gamma_1+\gamma_2}(1+p)^{a_1 + a_2}
\end{pmatrix} \]

\noindent and the image of $Z$ on $\sigma_1\sigma_2$ is 
$$Z_{\sigma_1\sigma_2} = \begin{pmatrix}
  0 \\
  (c_1 + \lambda^{\gamma_1}c_2)p^2\\
\end{pmatrix} .$$ 

\noindent On the other hand,
\[Z_{\sigma_1} + \sigma_1 Z_{\sigma_2} =
\begin{pmatrix}0\\ c_1 p^2 \end{pmatrix} + \begin{pmatrix}
(1+p)^{a_1} & 0\\
c_1p^2 & \lambda^{\gamma_1}(1+p)\\
\end{pmatrix} \begin{pmatrix} 0\\  c_2 p^2 \end{pmatrix}  \equiv \begin{pmatrix}
  0 \\
  (c_1 + \lambda^{\gamma_1}c_2)p^2\\
\end{pmatrix} \mod p^3.
\]

\noindent Therefore $Z$ represents the class of a cocycle in $\Hh^1(G_3, (\Z/p^3\Z)^2)$.
We are going to show that $Z$ satisfies the local conditions, i.e.\ that the equation
$$(\sigma-\Id)\begin{pmatrix}x \\ y \end{pmatrix}=\begin{pmatrix} 0 \\ cp^2 \end{pmatrix}$$
admits a solution, for all $a,c, \gamma$. This yields to the following system of equations
\[\begin{cases}
    (ap+\binom{a}{2}p^{2}) x\equiv 0 \mod p^3\\
    c p^2 x + (\lambda^\gamma-1+\lambda^\gamma ap +\lambda^\gamma\binom{a}{2}p^2) y \equiv c p^2 \mod p^3 .
\end{cases}\]

\noindent If $\lambda^\gamma-1+\lambda^\gamma ap +\lambda^\gamma\binom{a}{2}p^2$ is an invertible element in
$\ZZ/p^3\ZZ$, then a solution is

\[
\begin{cases}
    x=0\\
    y =  c (\lambda^\gamma-1+\lambda^\gamma ap +\lambda^\gamma\binom{a}{2}p^2)^{-1}  p^2.
\end{cases}
\]

\noindent Suppose that $p|\lambda^\gamma-1+\lambda^\gamma ap +\lambda^\gamma\binom{a}{2}p^2$. If
$\lambda^\gamma-1+\lambda^\gamma ap +\lambda^\gamma\binom{a}{2}p^2=p\omega$, with $\omega\in (\ZZ/p^3\ZZ)^*$, then
a solution is
\[
\begin{cases}
    x=0\\
    y =  c \omega^{-1}  p.
\end{cases}
\]

\noindent If
$\lambda^\gamma-1+\lambda^\gamma ap +\lambda^\gamma\binom{a}{2}p^2=\eta p^2$, with $\eta\in (\ZZ/p^3\ZZ)^*$, then
a solution is
\[
\begin{cases}
    x=0\\
    y =  c \eta^{-1}.
\end{cases}
\]

\noindent We are left with the case when $\lambda^\gamma-1+\lambda^\gamma ap +\lambda^\gamma\binom{a}{2}p^2\equiv 0 \mod p^3$, i.e.\ when
$\lambda^\gamma+\lambda^\gamma ap +\lambda^\gamma\binom{a}{2}p^2\equiv 1 \mod p^3$, which is equivalent to 
$\lambda^\gamma(1+p)^a\equiv 1 \mod p^3$. We have already observed that for our choice of $\lambda$, with order coprime
with the order of $1+p$, this may happen if and only if $\lambda^\gamma\equiv (1+p)^a\equiv 1\mod p^3$. In this last case we have 
$\delta^a\equiv \Id \mod p^3$, as well as $\rho\equiv \Id\mod p^3$. Therefore $\sigma=\tau_L^c$ and a solution of the
system is
\[
\begin{cases}
    x=1\\
    y = 0.
\end{cases}
\]
\noindent As in the previous cases, the solution depends on $a,c$ and $\gamma$, thus $Z$ is not a coboundary.
Anyway, to verify this directly, we can take $\sigma=\tau_L$ and $\sigma=\delta$. For $\sigma=\tau_L$, the
second equation in the system is $p^2 x\equiv p^2 \mod p^3$, implying $x\equiv 1\mod p$. For $\sigma=\delta$, the
first equation in the system is $p x\equiv 0 \mod p^3$, implying $x\equiv 0\mod p^2$. 

\medskip
 We are left with the case where $n=2$. The case when $j>m$ does not hold, because of the assumptions $1\leq m<2$ and $1\leq j<2$.  
    Then the only case left is when $j=m=1$, for which we have the mentioned example produced by Ranieri in \cite[Lemma 10]{Ranieri2018CounterexamplesCurves}.

\par\medskip As in the case where $j<m$, also here for $j\geq m$ we have that, by \cite[Lemma 11]{Gillibert2017OnVarieties}, there exists a number field $k$ and an elliptic curve $\mathcal E$ over $k$ such that $\Gal\left( k\left(\mathcal{E}[p^n]\right)/k \right)$ is isomorphic to each of the groups $G_n$ as above. Then in particular $\Hh^1_\loc (G_n,\E[p^n])$ is not trivial. 
\end{proof}

\begin{remark}
As in Remark \ref{rmk:lambda}, for $n\geq 5$, one can obtain other counterexamples by choosing $\lambda = \alpha + p^{h+s}{\theta}$, with $2 \leq s < n/2$, for $n$ even, and $2 \leq s < (n-1)/2$ if $n$, for $n$ odd (in order to have $h+s < n$). For $n=4$, we have that $p^h$ is already as maximum as possible, because
of $h+2=4$ and in this case $\lambda=\alpha$. Similarly, for $n=3$, we have $h+2=3$.
\end{remark}

\begin{remark} \label{rem_criterium}
Observe that for $n=3$ and $j\geq m$, the groups $G_3$ are formed by matrices in lower triangular form that reduced modulo $p^2$ are diagonal. Then $G_2$ is diagonal in this case and
the local-global principle for divisibility fails. Thus the hypothesis of Corollary \ref{criterium} that $G_2$ is lower triangular form cannot be improved further since for groups $G_2$ in diagonal form, counterexamples appear as well.
\end{remark}

\noindent As a consequence of Theorem \ref{thm:counterexamples},  we are going to show that all the counterexamples we produced in Theorem \ref{thm:counterexamples} for the local-global divisibility by $p^n$ in $\E$, give counterexamples to the local-global divisibility by $p^{n+s}$ in $\E$ over a finite extension $L_s$ of $k$, for every integer $s\geq 0$.

\begin{corollary} \label{n+s}
Let $n\geq 2$, $s\geq 0$ be integers. Let $p\geq 5$ be a prime number.
For every elliptic curve $\E$ satisfying the hypotheses of Theorem \ref{thm:counterexamples}, 
there exists a point $P\in \E(L_s)$, with $L_s$ a finite extension of $k$, such that $P$ 
is locally divisible by $p^{n+s}$ in $\E((L_s)_w)$, for all but finitely many 
$w\in M_{L_s}$ (where $(L_s)_w$ is the completion of $L_s$ at $w$), but $P$ is not divisible by $p^{n+s}$ in $\E(L_s)$.
\end{corollary}

\begin{proof}
For $s=0$, by \cite[Theorem 3]{Dvornicich2007OnInteger}  the nontriviality of $\Hh^1_\loc (G_n,\E[p^n])$, proved in Theorem \ref{thm:counterexamples}, implies that the local-global divisibility by $p^n$ does not hold in
$\E$ over a finite extension $L_0$ of $k$. Observe that in all the counterexamples produced in the proof of  Theorem \ref{thm:counterexamples} when $n\geq 3$, for every $\sigma\in G_n$ we have $p Z_\sigma=(0,0)$, which implies $Z_\sigma\in \E[p]=\E[p^{n-t}]$, with $t=n-1$. Then one of the hypotheses of \cite[Theorem 2.1]{Paladino2012OnGroups} is satisfied with $t=n-1$. In order to apply \cite[Theorem 2.1]{Paladino2012OnGroups}  we need to show in addition that $\E$ has no $k$-rational points of exact order $p^{t+1}=p^{n}$. 
 In the case when $j<m$ and $n\geq 3$, the group $G_n$ giving the counterexample in the proof of Theorem \ref{thm:counterexamples} is generated by the following automorphisms 
$$\tau_L=\begin{pmatrix}
1 & 0\\
p^{n-2} & 1\\
\end{pmatrix},
\quad
\tau_U=\begin{pmatrix}
1 & p^{i}\\
0 & 1\\
\end{pmatrix},
\quad
\delta=\begin{pmatrix}
1 +p^{n-1}& 0\\
0 & 1+p^{h}\\
\end{pmatrix},
\quad
\rho=\begin{pmatrix}
1 & 0\\
0 & \lambda\\
\end{pmatrix}.$$
If $P=(x,y)\in \E[p^n]$ is $k$-rational, then $\sigma(P)=P$, for every $\sigma\in G_n$.
By the generators
as above we in particular get the equations  $y+ p^h y \equiv y \mod p^n$, i.e.\ $p^h y\equiv 0 \mod p^n$, and $p^{n-2}x+y \equiv y \mod p^n$, 
i.e.\  $p^{n-2} x\equiv 0 \mod p^n$.
The equation $p^h y\equiv 0 \mod p^n$ implies $p^{n-h}|y$ and in particular $p^2| y$, by the definition of $h$. The congruence $p^{n-2}x\equiv 0 \mod p^n$ also implies $p^2|x$ and then we have that $p^{n-2}P=(0,0)$.
Thus every point in $\E[p^n]$ fixed by $G_n$ lies indeed in $\E[p^{n-2}]$ and it does not have exact order $p^n$.
Therefore we can apply \cite[Theorem 2.1]{Paladino2012OnGroups} with $t=n-1$ to get the conclusion. For $j\geq m$ and $n\geq 3$ the proof is very similar with the only difference that we can consider the equations  $p^h y\equiv 0 \mod p^n$ and $p^{n-1} x\equiv 0\mod p^n$,  implying
$P\in \E[p^{n-1}]$  (observe that $h=1$, when $p=3$).  Again $P$ has not exact order $p^n$ and we can apply \cite[Theorem 2.1]{Paladino2012OnGroups} with $t=n-1$.
 For $n=2$, we consider the example produced in \cite[Lemma 10]{Ranieri2018CounterexamplesCurves}. We have that the cocycle whose class is a nontrivial element in the first cohomology groups has values in $\E[p]$. By considering the matrices in $G_2$ one can deduce that if $P=(x,y)$ is a $k$-rational point of order $p^2$, then $x\equiv 0\mod p$ and $y\equiv 0\mod p$,
implying that $P$ has order $p$ indeed. Hence one can apply \cite[Theorem 2.1]{Paladino2012OnGroups} with $t=n-1=1$ again.
\end{proof}

\begin{remark}
Observe that even without taking into account Corollary \ref{n+s},
Theorem \ref{thm:counterexamples} proves that for every power $p^n$, with $p\geq 5$ and $n\geq 2$  there exist of a number field $k$ and an elliptic curve $\E$ defined over $k$, such that the local-global divisibility by $p^n$ fails in $\E$ over $k$. In fact,  groups $G_n$ such that $\Hh^1_\loc (G_n,\E[p^n])\neq 0$ are showed for every $n\geq 2$ and every $p\geq 5$ and by \cite[Theorem 3]{Dvornicich2007OnInteger} this implies the failing of the Hasse principle for divisibility by $p^n$ in $\E$ over $k$. Anyway Corollary \ref{n+s} shows the failing of the principle in the \emph{same} elliptic curve $\E$, for all powers $p^s$, with $s\geq n$, whenever $n\geq 2$ and $p\geq 5$. 
\end{remark}

\begin{remark} \label{isogeny}
We relate both cases of Theorem \ref{thm:counterexamples}, when $j<m$ and when $j\geq m$, to the existence or non-existence in $\E$ of $k$-rational cyclic isogenies of degrees a power of $p$.
\begin{enumerate}
\item In the cases when $j<m$, observe that if $n>3$ 
and $i<n$, then the elliptic curve $\E$ admits a $k$-rational cyclic isogeny of degree $p^l$, for all $1\leq l\leq n-2$, but does not admit a cyclic isogeny of degree $p^{n-1}$ and the local-global divisibility by $p^n$ does not hold. 
In fact, the matrices in $G_n$ reduce to matrices in upper triangular form modulo $p^l,$ for all $1\leq l\leq n-2$, while modulo $p^n$, the matrices are neither in upper triangular nor in lower triangular (again with respect to the fixed basis $\{Q_1,Q_2\}$), since $h+1 < i < n$ and $h=n/2$ if $n$ even and $h = (n-1)/2$ if $n$ odd. Observe that this also implies that
$\E$ does not admit a cyclic isogeny of degree $p^{n}$. 
\item Similarly, when $j \geq m$, if $n>3$
and $i<n$, the elliptic curve $\E$ admits a $k$-rational cyclic isogeny of degree $p^l$, for all $1\leq l\leq n-1$, but does not admit a cyclic isogeny of degree $p^n$ and the local-global divisibility by $p^n$ does not hold.
Indeed, we have that the matrices in $G_n$ modulo $p^l$ reduce to matrices in upper triangular form for $1\leq l\leq n-1$, while modulo $p^n$ the matrices are neither in upper triangular nor in lower triangular form (again with respect to the fixed basis $\{Q_1,Q_2\}$), as $h+1 < i < n$ and $h=n/2$ if $n$ even and $h = (n+1)/2$ if $n$ odd.
\end{enumerate}

\noindent 
This is somewhat unexpected. In fact, by \cite[p.~28]{Dvornicich2007OnInteger}, the non-existence of a cyclic $k$-rational isogeny of degree $p$ assures the validity of the local-global principle for divisibility by $p^n$ for every $n\geq 1$. Here instead we have showed that the non-existence of an isogeny of degree $p^n$ does not imply the validity of the local-global divisibility by $p^s$, for all $s\geq n$.
\end{remark}

\bibliographystyle{alpha}
\bibliography{ref}

\bigskip

\begin{minipage}[t]{10cm}
	\begin{flushleft}
		\small{
			\textsc{Jessica Alessandr\`i}
			\\* Max Planck Institute for Mathematics,
			\\* Vivatsgasse 7,
			\\* 53111 Bonn, Germany
			\\*e-mail: alessandri@mpim-bonn.mpg.de
				
		}
	\end{flushleft}
\end{minipage}

\bigskip
	
\begin{minipage}[t]{10cm}
    \begin{flushleft}
	    \small{
		    \textsc{Laura Paladino}
   		    \\*Universit\`a della Calabria,
   		   	\\* Ponte Bucci, Cubo 30B 
	       	\\* Rende (CS), 87036, Italy
	    	\\*e-mail: laura.paladino@unical.it
			
    	}
    \end{flushleft}
\end{minipage}

\end{document}